\documentclass[11pt]{amsart}

\usepackage{amssymb,amsmath,amsthm}
\usepackage{enumerate,titletoc}
\usepackage[pagebackref,colorlinks,linkcolor=red,citecolor=blue,urlcolor=blue,hypertexnames=true]{hyperref}
\makeindex

\DeclareMathOperator{\tr}{tr}

\DeclareMathOperator{\QM}{QM}

\renewcommand{\span}{{\rm span}}

\renewcommand{\emptyset}{\varnothing}
\renewcommand{\setminus}{\smallsetminus}

\def\tg{g}

\def\Rdd{\R^{d \times d} }

\def\mbS{\mathbb S }
\def\SRdd{\mathbb S\R^{d \times d} }
\def\SRnn{\mathbb S\R^{n \times n} }

\def\T{T}
\def\bes{\begin{equation*}}
\def\ees{\end{equation*}}

\def\beq{\begin{equation} }
\def\eeq{\end{equation} }

\def\ep{\varepsilon}

\def\ncboL{{non-commutative basic open\ }}
\def\ncbboL{{non-commutative bounded basic open\ }}
\def\ncbo{{non-commutative basic open\ }}
\def\ncbbo{{non-commutative bounded basic open\ }}

\def\cC{\mathcal C}
\def\cD{\mathcal D}

\def\cL{\mathcal L}

\def\cS{\mathcal S}

\def\cI{\mathcal I}

\def\ccP{\R\ax}

\def\cT{\mathcal T}

\def\smatng{(\SRnn)^g}

\def\posn{\mathfrak{I}_p(n)}

\def\pos{\mathfrak{I}_p}

\def\cDpn{\cD_p(n)}

\def\tL{\tilde L}

\def\cDpn{\cD_p(n)}

\def\dd{\delta}
\def\dddd{\delta \times \delta}
\def\ddp{\dd^\prime}

\def\nus{\breve \nu}

\def\ocD{\overline{\cD}}
\def\Nd{ {\left\lceil\frac{d}{2}\right\rceil }}

\def\ben{\begin{enumerate} }
\def\een{\end{enumerate} }
\def\benum{\begin{enumerate} }
\def\eenum{\end{enumerate} }

\def\x{x}
\def\xs{x^*}
\def\y{y}
\newcommand{\ax}{\langle\x\rangle}
\newcommand{\axs}{\langle\x,\xs\rangle}
\newcommand{\cy}{[\y]}

\def\bmat{\left[\begin{array}{ccccccccccccccc} }
\def\emat{\end{array}\right]}

\def\bmat{\begin{bmatrix}}
\def\emat{\end{bmatrix}}
\def\beq{\begin{equation}}
\def\eeq{\end{equation}}

\def\barr{\begin{array}}
\def\earr{\end{array}}

\def\T{\ast}
\def\TT{\ast}

\def\Rnn{ \R^{n \times n} }
\def\NN{\mathbb N}
\def\N{\mathbb N}

\def\FF{\CCxxS}

\def\smatn{\bbS \R^{n \times n}  }

\def\gtupn{(\smatn)^g}

\def\bbS{{\mathbb S}}

\def\cA{ {\mathcal A} }

\def\cC{ {\mathcal C} }
\def\cD{ {\mathcal D} }

\def\cI{ {\mathcal I} }

\def\cL{ L }

\def\cS{{\mathcal S} }
\def\cT{{\mathcal T}}

\def\la{\lambda}

\def\tL{{\tilde L}}

\def\smatng{(\SRnn)^g }

\def\j1tog{ j= 1, \ldots, g }

\def\bbS{ {\mathbb S}}

\def\cI{{\mathcal I}}

\def\NC{non-commutative}

\def\tA{{\widetilde A}}

\def\tA{\tilde{A}}

\def\L0t{\ (L_0 \otimes I_n ) \ }

\def\PLINE1{\beta}

\newcommand{\C}{{\mathbb C}}

\newcommand{\RR}{{\mathbb R}}
\newcommand{\R}{{\mathbb R}}

\def\NCRAG{ Non-Commutative Semi-Algebraic Geometry }

\def\NC{non-commutative }

\def\pp{q}
\def\qq{p}

\def\RRx{\RR \langle x \rangle}

\def\RRxxS{\RR\langle x,x^\T \rangle}

\def\FF{\RRxxS}

\def\possss{Positivstellens\"atze}

\newtheorem{thm}{Theorem}[section]

\newtheorem{cor}[thm]{Corollary}

\newtheorem{prop}[thm]{Proposition}
\theoremstyle{definition}

\theoremstyle{remark}

\newtheorem{remark}[thm]{Remark}

\newtheorem{assumption}     [thm]{Assumption}

\numberwithin{equation}{section}

\newtheorem{exa}[thm]{Example}
\newenvironment{example}%
         {\begin{exa}}
             {{\hfill $\Box ~~$}\end{exa}}

\newcounter{Inc}

\begin{document}
\setcounter{page}{1}

\title[SDP in Matrix Unknowns]{
Semidefinite programming in matrix unknowns which are dimension free}

\author[Helton]{J. William Helton${}^1$}
\address{J. William Helton, Department of Mathematics\\
  University of California \\
  San Diego}
\email{helton@math.ucsd.edu}
\thanks{${}^1$Research supported by NSF grants
DMS-0700758, DMS-0757212, and the Ford Motor Co.}

\author[Klep]{Igor Klep${}^2$}
\address{Igor Klep, Univerza v Ljubljani, Fakulteta za matematiko in fiziko \\
and
Univerza v Mariboru, Fakulteta za naravoslovje in matematiko 
}
\email{igor.klep@fmf.uni-lj.si}
\thanks{${}^2$Research supported by the Slovenian Research Agency grants
P1-0222 and P1-0288.}

\author[McCullough]{Scott McCullough${}^3$}
\address{Scott McCullough, Department of Mathematics\\
  University of Florida 
   }
   \email{sam@math.ufl.edu}
\thanks{${}^3$Research supported by the NSF grant DMS-0758306.}

\date{21 June 2010}

\setcounter{tocdepth}{3}
\contentsmargin{2.55em} 
\dottedcontents{section}[3.8em]{}{2.3em}{.4pc} 
\dottedcontents{subsection}[6.1em]{}{3.2em}{.4pc}
\dottedcontents{subsubsection}[8.4em]{}{4.1em}{.4pc}

\makeatletter
\newcommand{\mycontentsbox}{%
{\centerline{NOT FOR PUBLICATION}
\linespread{1.2}
\small\tableofcontents}}
\def\enddoc@text{\ifx\@empty\@translators \else\@settranslators\fi
\ifx\@empty\addresses \else\@setaddresses\fi
\newpage\mycontentsbox}
\makeatother

\begin{abstract}
One of the main applications of semidefinite programming lies
in linear systems and control theory. Many problems in this subject,
certainly the textbook classics, have matrices as variables,
and the  formulas naturally contain non-commutative polynomials in
matrices. These polynomials depend only on the system layout
and do not change with the size of the matrices involved,
hence such problems are called ``dimension-free''.
Analyzing dimension-free problems has led to the development 
recently of a non-commutative (nc) real algebraic geometry (RAG)
which, when combined with convexity, produces 
 dimension-free Semidefinite Programming.
This article surveys what is known about convexity in the non-commutative
setting and nc SDP and includes a brief survey of nc RAG.
Typically, the qualitative properties of the non-commutative case
are much cleaner than those of their scalar 
counterparts - 
variables in $\RR^g$.
Indeed we describe how relaxation of scalar variables
by matrix variables in several natural situations
results in a beautiful structure.
\end{abstract}

\maketitle 

\section{Introduction}
 Given symmetric $\ell \times\ell$ symmetric matrices
 with real entries $A_j$, the expression
\begin{equation}
 \label{eq:lmi}
  L(x)= I_\ell + \sum_{j=1}^g A_j x_j \succ 0
\end{equation}
 is a {\bf linear matrix inequality} (LMI).  Here 
  $\succ 0$ means positive definite, 
  $x=(x_1,\dots,x_g)\in\mathbb R^g$ and of interest is the set
 of solutions $x$.
 Taking advantage of the Kronecker (tensor) product $A\otimes B$ 
 of matrices, it is natural to consider, for tuples of
symmetric $n\times n$
 matrices $X=(X_1,\dots,X_g)\in (\mbS\R^{n\times n})^g$, the inequality
\begin{equation}
 \label{eq:LMI}
   L(X)= I_\ell \otimes I_n + \sum_{j=1}^g A_j\otimes X_j \succ 0.
\end{equation}
  For reasons which will become apparent soon,
  we call expression \eqref{eq:LMI} a {\bf non-commutative LMI} (nc LMI). 
  Letting $\cD_L(n)$ denote the solutions $X$ of size $n\times n$, note
  that $\cD_L(1)$ is the solution set of equation \eqref{eq:lmi}.
  In many areas of mathematics and its applications, the
  inequality \eqref{eq:LMI} is called the {\bf quantized} 
  version of inequality \eqref{eq:lmi}.

  Quantizing a polynomial inequality requires the notion of 
  a non-commutative (free) polynomial which 
  can loosely be thought of as a polynomial in matrix unknowns.
  Section \ref{subsec:ncpolys} below gives the details on these
  polynomials. For now we limit the discussion to the example,
\begin{equation}
 \label{eq:ex-sym-p}
  p(x,y) = 4 - x-y -(2x^2 +xy+yx+2y^2).
\end{equation}
  Of course, for symmetric $n\times n$ matrices $X,Y$, 
\begin{equation}
 \label{eq:ex-p-eval}
  p(X,Y)=4I_n-X-Y-(2X^2+XY+YX+2Y^2).
\end{equation}
  The set $\{(x,y)\in\mathbb R^2:p(x,y)>0\}$ is a semi-algebraic set. 
  By analogy, the set $\{(X,Y): p(X,Y)\succ 0\}$ is
  a {\bf non-commutative semi-algebraic set}.
 
  nc LMIs, and more generally non-commutative semi-algebraic sets,
  arise naturally in semidefinite programming (SDP) and in linear
  systems theory problems determined by a signal-flow diagram.
  They are of course basic objects in the study of operator
  spaces and thus are related to problems like Connes' embedding
  conjecture \cite{Con,KS08} and the  Bessis-Moussa-Villani (BMV) conjecture \cite{bmv} 
  from quantum statistical mechanics \cite{KSbmv}.  As is seen
  in Theorem \ref{thm:commposIntro} below, they even have something
  to say about their scalar (commutative) counterparts. For some these
  non-commutative considerations have their own intrinsic interest as
  a free analog to classical semi-algebraic geometry. 

 Non-commutative will often be shortened to nc.

\subsection{The roadmap}
The paper treats four   areas of research concerning
  nc LMIs and  nc polynomials.

  In the remainder of the introduction we first,
  in Subsection \ref{subsec:ncpolys},  
  give additional background on 
  a core object of our study, polynomials in non-commuting variables.
  The initiated reader may wish to skip this subsection. 
 Subsections \ref{subsec:intro-to-domination}, \ref{subsec:nc-convex-sets},
  \ref{subsec:convex-polys}, and \ref{subsec:alg-certs} give 
 overviews of the four main topics of the survey. 

 The body of the paper consists of six sections. 
 The first four give further detail on the main topics. 
 Except for Section \ref{sec:conv=LMI} which has its own
 motivation subsection, motivation for our investigations is
 weaved into the discussion. Convexity is a recurring theme.
 Section \ref{sec:softw} 
 offers a list of computer algebra packages for work in a free $*$-algebra
 revolving around convexity and positivity.

\subsection{Non-commutative polynomials}
  \label{subsec:ncpolys}
   Let
   $\ccP$ \index{$\ccP$} denote the real algebra of polynomials in the
   non-commuting indeterminates $x=(x_1,\ldots,x_g).$
   Elements of $\ccP$ are
   {\bf non-commutative polynomials}, abbreviated to
  {\bf nc polynomials} or often just {\bf polynomials}.
   \index{non-commutative polynomial}
  \index{nc polynomial}
   Thus, a non-commutative polynomial $p$ is a finite sum,
 \begin{equation}
  \label{eq:poly}
    p=\sum p_w w,
 \end{equation}
   where each $w$ is a word in $(x_1,\ldots,x_g)$ and the coefficients
 $p_w\in\mathbb{R}$. The polynomial $p$ of equation \eqref{eq:ex-sym-p}
 is a non-commutative polynomial of degree two in two variables.
 The polynomial
\begin{equation}
 \label{sym-and-not}
  q= x_1 x_2^3 + x_2^3 x_1 + x_3x_1x_2 +x_2 x_1 x_3
\end{equation}
 is an non-commutative polynomial of degree four in three variables.

\subsubsection{Involution}
   There is a natural {\bf involution} ${}^*$ on
   $\ccP$  given by
 \begin{equation}
  \label{eq:genericp}
   p^*=\sum p_w w^*,
 \end{equation}
  where, for a word $w$,
 \begin{equation}
  \label{eq:wordT}
   w=x_{j_1}x_{j_2}\cdots x_{j_n} \mapsto w^* = x_{j_n}
   \cdots x_{j_2}x_{j_1}.
 \end{equation}
   A polynomial $p$ is {\bf symmetric} if $p^*=p$. \index{symmetric polynomial}
   For example, the polynomials of
   equation \eqref{eq:ex-sym-p} is symmetric, whereas the
   $q$ of equation \eqref{sym-and-not} is not. 
   In particular, $x_j^*=x_j$ and for this reason
   the variables are sometimes referred to as symmetric
   non-commuting variables. \index{symmetric variables}
 
  Denote, by $\ccP_d,$ the polynomials in $\ccP$ of (total) degree $d$
  or less.

\subsubsection{Substituting Matrices for Indeterminates}
  Let $\smatng$ denote the set of $g$-tuples \index{$\smatng$}
  $X=(X_1,\ldots,X_g)$ of real symmetric $n\times n$ matrices.
  A  polynomial $p(x)=p(x_1,\ldots,x_g)\in \ccP$
  can naturally be evaluated at a
  tuple $X\in \smatng$
  resulting in  an $n\times n$ matrix. 
  Equations \eqref{eq:ex-sym-p} and \eqref{eq:ex-p-eval}
  are illustrative.  
  In particular,
  the constant term  $p_\emptyset$ of $p(x)$  becomes $p_\emptyset I_n$;
  i.e., the empty word evaluates to $I_n$.
  Often we write $p(0)$ for $p_\emptyset$ interpreting the
  $0$ as $0\in\mathbb R^g$.
 As a further 
 example, for the polynomial $q$ from equation \eqref{sym-and-not},
\[
  q(X)= X_1 X_2^3 + X_2^3 X_1 + X_3X_1X_2 +X_2 X_1 X_3.
\]

  The involution on $\ccP$
  that was introduced earlier is compatible with
  evaluation at $X$ and matrix transposition, i.e.,
\[
 p^*(X)=p(X)^*,
\]
  where $p(X)^*$ denotes the transpose of the  matrix  $ p(X)$.
  Note, if $p$ is symmetric, then so is $p(X)$.

\subsubsection{Matrix-Valued Polynomials}
  Let $\ccP^{\dd\times \ddp}$ denote \index{$\ccP^{\dd\times \ddp}$}
  the $\dd\times \ddp$ matrices with entries
  from $\ccP$.  In particular, if $p\in \ccP^{\dd\times \ddp}$, then
 \begin{equation}
  \label{eq:pww}
   p=\sum p_w w,
\end{equation}
  where the sum is finite and each $p_w$ is a real $\dd\times \ddp$  matrix.
  Denote, by $\ccP_d^{\dd\times \ddp},$
 the subset of $\ccP^{\dd\times \ddp}$
 whose polynomial entries have degree $d$ or less.
 \index{$\ccP^{\dd\times \ddp}_m$}

  Evaluation   at $X \in \smatng$ naturally
  extends to $p\in \ccP^{\dd\times\ddp}$
via the Kronecker tensor product,
  with the result, $p(X),$ a $\dd\times\ddp$ block
  matrix with $n\times n$ entries.
 The involution ${}^*$ naturally extends to
  $\ccP^{\dddd}$ by
 \begin{equation}
  \label{eq:pwwstar}
   p=\sum p_w^* w^*,
\end{equation}
  for $p$ given by equation \eqref{eq:pww}.
  A polynomial
  $p\in\ccP^{\dd\times \dd}$ is symmetric if
  $p^*=p$ and in this case $p(X)=p(X)^*$.

  A simple method of constructing new matrix valued
  polynomials from old ones is by direct sum. \index{direct sum}
  For instance, if
  $p_j\in\ccP^{\dd_j\times \dd_j}$ for $j=1,2$, then
 \[
   p_1\oplus p_2=\begin{bmatrix} p_1 & 0\\0 & p_2 \end{bmatrix}
    \in \ccP^{(\dd_1+\dd_2)\times (\dd_1+\dd_2)}.
 \]
   \index{$p_1\oplus p_2$}

\subsubsection{Linear Matrix Inequalities $($LMIs$)$}
  Given symmetric $\ell \times \ell$ 
  matrices $A_0,A_1,\dots,A_g$, the expression
 \beq\label{eq:pencilDef}
    L(x) = A_0 + \sum_{j=1}^g A_j x_j
 \eeq
  is an affine linear nc matrix polynomial, better known as a {\bf linear} 
  (or affine linear) {\bf pencil}.  In the case that $A_0=0$, $L$
  is a {\bf truly linear pencil}; and when $A_0=I$, we say
  $L$ is a {\bf monic linear pencil}. 

  The inequality $L(x)\succ 0$ for $x\in\mathbb R^g$ is a
  {\bf linear matrix inequality (LMI)}. LMIs are ubiquitous in science
  and engineering. 
  Evaluation of $L$ at $X\in\smatng$ is most easily described 
  using tensor products as in equation \eqref{eq:LMI} and
  the expression $L(X)\succ 0$ is a {\bf non-commutative LMI},
  or nc LMI for short.   

\subsection{LMI Domination and Complete Positivity}
 \label{subsec:intro-to-domination}
  This section discusses the nc LMI versions of 
   two natural LMI domination questions. 
   To fix notation, let
\[
  L(x)=A_0+\sum_{j=1}^g A_j x_j,
\]
   be a given linear pencil 
  (thus $A_j$ are symmetric $\ell \times\ell$ matrices).
   For a fixed $n$ the solution set of all
$X\in\smatng$ satisfying $L(X)\succ 0$ is denoted $\cD_L(n)$ and the
  sequence (graded set) $(\cD_L(n))_{n\in\N}$ is written $\cD_L$. 
  Note that $\cD_L(1)$ is the solution set of the 
  classical (commutative) LMI, $L(x)\succ 0$.

  Given linear matrix inequalities  (LMIs) $L_1$ and $L_2$
  it is natural to ask:
\ben[\rm (Q$_1$)]
\item
 when does one dominate the other, that is,
 when is $\cD_{L_1}(1)\subseteq \cD_{L_2}(1)$? 
\item
 when are they mutually dominant, that is, $\cD_{L_1}(1)=\cD_{L_2}(2)$?
\een 
 While such problems can be NP-hard, their nc relaxations
 have elegant answers.
 Indeed, they reduce to constructible semidefinite programs.
 We chose to begin with this topic because it offers the most gentle
 introduction to our matrix subject.

 To describe a sample result,
 assume there is an $x\in\mathbb R^g$ such that 
 both $L_1(x)$ and $L_2(x)$
 are both positive definite,
 and suppose $\cD_{L_1}(1)$ is bounded. 
 If $\cD_{L_1}(n)\subseteq \cD_{L_2}(n)$ for every $n$, then
 there exist matrices $V_j$ such that
\beq\tag{\rm A$_1$}\label{eq:abstr}
L_2(x)=V_1^* L_1(x) V_1 + \cdots + V_\mu^* L_1(x) V_\mu.
\eeq
  The converse is of course immediate. 
 As for (Q$_2$) we show that $L_1$ and $L_2$
  are mutually dominant ($\cD_{L_1}(n)=\cD_{L_2}(n)$ for all $n$)
  if and only if, up to
 certain  redundancies described in detail in Section \ref{sec:lmiDom&CP},
 $L_1$ and $L_2$
 are unitarily equivalent.

 It turns out that our
 matrix variable LMI domination problem is equivalent
 to the  classical problem of determining
 if a linear map $\tau$ from one subspace of matrices to another is
 ``completely positive''. Complete positivity is one of the main techniques
 of modern operator theory and
 the theory of operator algebras. On one hand it provides tools for
 studying LMIs and on the other hand, since completely positive
 maps are not so far from representations and generally are more
 tractable than their merely positive counterparts, the theory of
 completely positive maps provides perspective on the
 difficulties in solving LMI domination problems.
 nc LMI domination is the topic of Section \ref{sec:lmiDom&CP}.

\subsection{Non-commutative Convex Sets and LMI Representations}
 \label{subsec:nc-convex-sets}
 Section  \ref{subsec:intro-to-domination} dealt with
 the (matricial) solution set of a
  Linear Matrix Inequality
\[
  \cD_L = \{X \colon  L(X) \succ 0 \}.
\]
 The set $\cD_L$ is convex in the
  sense that each $\cD_L(n)$ is convex.
  It is also a non-commutative basic open semi-algebraic set
 (in a sense we soon define).
 The main theorem of this section is the converse,
 a result
 which has implications for both semidefinite programming
 and systems engineering.

  Let $p\in \ccP^{\dd\times \dd}$ be a given symmetric non-commutative
  $\dd\times\dd$-valued matrix polynomial. 
  Assuming that $p(0)\succ 0$, the
 positivity   set $\cD_{p}(n)$ of
  a non-commutative symmetric polynomial $p$ in dimension $n$
  is the component of $0$ of the set
\[
   \{ X \in(\SRnn)^g  \colon
  \ p(X)\succ0 \}.
\]
  The
  {\bf positivity
  set}, $\cD_p$, is the
  sequence of sets $(\cD_p(n) )$,
  which is the type of set we call  a
  {\bf\ncbo} semi-algebraic set.
 The non-commutative set $\cD_p$ is  called {\bf convex} if, for each $n,$
  $\cD_p(n)$ is convex.
  A set is said to have a {\bf Linear Matrix Inequality Representation}
  if it is the set of all solutions to some LMI, that is,
  it has the form $\cD_L$ for some $L(x)=I+\sum_j A_jx_j$.

 The main theorem of Section \ref{sec:conv=LMI} says:
 if $p(0)\succ 0$ 
 and  $\cD_p$ is  bounded,
 then $\cD_{p}$ has an LMI representation
 if and only if $\cD_p$ is convex.

\subsection{Non-commutative Convex Polynomials have Degree Two}
 \label{subsec:convex-polys}
 We turn now from non-commutative convex sets to
 non-commutative convex polynomials.
 The previous section exposed the rigid the structure of
 sets which are both convex and the sublevel set of a non-commutative
 polynomial.  Of course if $p$ is concave ($-p$ is convex), then
 its sublevel sets are convex.  But more is true. 

  A symmetric polynomial $p$ is {\bf matrix convex},
  if for each positive integer $n$, each
  pair of tuples of symmetric matrices $X\in\smatng$ and 
  $Y\in\smatng$, and each $0\le t \le 1$,
\begin{equation*}
  p\big(tX+(1-t)Y\big)\preceq tp(X)+(1-t)p(Y).
\end{equation*}
  The main result on convex polynomials, given in Section \ref{sec:convexPolys}, is
   that every  symmetric non-commutative polynomial
   which is matrix convex 
    has degree two or less.

\subsection{Algebraic certificates of non-commutative positivity: \possss}
 \label{subsec:alg-certs}
 An algebraic certificate for positivity of a polynomial $p$
 on a semi-algebraic set $S$ is a Positivstellensatz. The familiar
 fact that a polynomial $p$ in one-variable which is positive on
 $S=\mathbb R$ is a sum of squares is an example. 

 The theory of Positivstellens\"atze -  a pillar of the
 field of semi-algebraic geometry -  underlies
 the main approach currently used for global optimization
 of polynomials. 
 See~\cite{P00,L00} for a beautiful treatment of this, and other,
 applications of
 commutative semi-algebraic geometry.
 Further, 
 because convexity of a polynomial
 $p$ on a set $S$ is equivalent to  positivity of the Hessian of
 $p$ on $S$, this theory also provides a link between convexity 
 and semi-algebraic geometry.  Indeed, this link in the non-commutative
 setting ultimately leads to the conclusion the a matrix convex non-commutative polynomial has degree at most two.

Polynomial optimization problems involving non-commuting variables also arise naturally in many areas of quantum physics, see \cite{Pi:09,Pi:10}.

 Positivstellens\"atze in various incarnations appear
 throughout this survey as they arise naturally in
 connection with the previous topics. 
 Section \ref{sec:posss&nullss} contains a brief list of 
 algebraic certificates for positivity like conditions
 for non-commutative polynomials in both symmetric and non-symmetric
 nc variables, Thus, this section
 provides an overview of non-commutative semi-algebraic geometry with 
 the theme being that nc Positivstellens\"atze are cleaner and more
 rigid than there commutative counterparts.

\section{LMI Domination and Complete Positivity}
\label{sec:lmiDom&CP}
 In this section we expand upon the discussion of nc LMI 
  domination of
 Subsection \ref{subsec:intro-to-domination}. Recall, a monic linear
 pencil is an expression of the form
\[
  L(x)=I+\sum_{j=1}^g A_j x_j,
\]
  where, for some $\ell$, the $A_j$ are symmetric $\ell \times \ell$ matrices
 with real entries and $I$ is the $\ell\times \ell$ identity. For a given 
 positive integer $n$, 
\[
  \cD_L(n)=\{X\in\smatng : L(X)\succ 0\}
\]
  and let $\cD_L$ denote the sequence of sets $(\cD_L(n))_{n\in\N}$. 
Thus $\cD_L$ is the solution set of the
 nc LMI $L(X)\succ 0$ and $\cD_L(1)$ is the solution set of the
 traditional LMI $L(x)\succ 0$ ($x\in\mathbb R^g$). 
We call $\cD_L$ an {\bf nc LMI}.

\subsection{Certificates for  LMI Domination}
 This subsection contains precise algebraic
 characterizations of nc LMI domination. Algorithms, the
 connection to complete positivity, examples, and the application to
 a new commutative Positivstellensatz 
 follow in succeeding subsections. 

\begin{thm}[Linear Positivstellensatz \cite{HKMlmi}]\label{thm:intro1}
 Let $L_j\in \mbS\R^{d_j\times d_j}\ax$, $j=1,2$,
 be monic linear pencils and assume
 $\cD_{\cL_1}(1)$ is bounded.
 Then $\cD_{L_1}\subseteq\cD_{L_2}$ if
 and only if there is a $\mu\in\NN$ and an 
 isometry $V\in\R^{\mu d_1\times d_2}$ such that
\beq\label{eq:linPosIntro}
  L_2(x)= V^* \big( I_\mu \otimes L_1(x) \big) V=
    \sum_{j=1}^\mu V_j^* L_1(x) V_j.
\eeq
\end{thm}

 Suppose $L\in\mbS\R^{d\times d}\ax$,
\[
  L= I +\sum_{j=1}^g A_j x_j
\]
 is a monic linear pencil.  A subspace $\mathcal H\subseteq \mathbb R^d$
 is {\bf reducing for $L$} if $\mathcal H$ reduces each
 $A_j$; i.e., if $A_j\mathcal H\subseteq \mathcal H$.  Since
 each $A_j$ is symmetric, it also follows that
 $A_j\mathcal H^\perp \subseteq \mathcal H^\perp$.  Hence, with respect
  to the decomposition $\mathbb R^d =\mathcal H\oplus \mathcal H^\perp$,
  $L$ can be written as the direct sum,
 \[
   L = \tL \oplus \tL^\perp
   =\begin{bmatrix} \tL & 0 \\ 0 & \tL^\perp \end{bmatrix}
\qquad \text{where} \qquad
  \tL= I+\sum_{j=1}^g \tA_j x_j,
\]
  and  $\tA_j$ is the restriction of $A_j$ to $\mathcal H$.
  (The pencil $\tL^\perp$ is defined similarly.)
  If $\mathcal H$ has dimension $\ell$, then by identifying
  $\mathcal H$ with $\mathbb R^\ell$, the pencil
  $\tL$    is a monic linear pencil of size $\ell$.
  We say that $\tL$ is a {\bf subpencil} of $L$.
  If moreover, $\cD_L=\cD_{\tL}$, then $\tL$
  is a {\bf defining subpencil} and if
  no proper subpencil of $\tL$ is a defining
  subpencil for $\cD_L$, then $\tL$ is
  a {\bf minimal defining $($sub$)$pencil}.

\begin{thm}[Linear Gleichstellensatz \cite{HKMlmi}]
 \label{thm:minimalIntro}
Suppose
$L_1, L_2$ 
are monic linear pencils with $\cD_{\cL_1}(1)$ bounded.
Then $\cD_{\cL_1}=\cD_{\cL_2}$
if and only if
  minimal defining pencils $\tL_1$ and $\tL_2$
for $\cD_{\cL_1}$ and $\cD_{\cL_2}$ respectively,
are unitarily equivalent. That is,
  there is a unitary  matrix
 $U$ such that
\beq\label{eq:Intro2}
\tL_2(x)=U^* \tL_1(x) U.
\eeq
\end{thm}

\subsection{Algorithms for LMIs}

Of widespread interest is determining if
\beq\label{eq:Q1}
\cD_{L_1}(1)
\subseteq
\cD_{L_2}(1),
\eeq
or if $\cD_{L_1}(1)=\cD_{L_2}(1)$.
For example, the paper of Ben-Tal and Nemirovski
\cite{B-TN} exhibits simple cases where determining this
is NP-hard.
While we do not give details here we guide the reader to
\cite[Section 4]{HKMlmi}
where
we prove that $\cD_{L_1}\subseteq\cD_{L_2}$ is equivalent
to the feasibility of a certain semidefinite program
which we construct explicitly in \cite[Section 4.1]{HKMlmi}.
Of course,
if $\cD_{L_1}\subseteq\cD_{L_2}$, then
$\cD_{L_1}(1)\subseteq\cD_{L_2}(1)$. Thus our algorithm is a
type of
relaxation of the problem \eqref{eq:Q1}.

Also in \cite{HKMlmi} is
an algorithm (Section 4.2)
easily adapted from the first to determine if
$\cD_L$ is bounded, and
what its ``radius'' is.
By \cite[Proposition 2.4]{HKMlmi}, $\cD_L$ is bounded
if and only if $\cD_L(1)$ is bounded.
Our algorithm thus yields an upper bound of the radius
of $\cD_L(1)$.
In 
 \cite[Section 4.3]{HKMlmi}  we solve a
matricial relaxation of the classical matrix
cube problem, finding the biggest matrix cube contained
in $\cD_L$. Finally,
given a matricial LMI set $\cD_L$,
\cite[Section 4.4]{HKMlmi}
gives an algorithm to compute the
linear pencil $\tL\in\mbS\Rdd\ax$ with smallest possible $d$ satisfying
$\cD_L=\cD_\tL$.

\subsection{Complete Positivity and LMI Inclusion}
\label{sec:cp&lmi}
 To monic linear pencils $L_1$ and $L_2,$ 
\beq
\label{eq:penc12}
 L_j(x) = I+\sum_{\ell=1}^\tg A_{j,\ell} x_\ell \in\mbS\R^{d_j\times d_j}\ax , \quad j=1,2
\eeq
 are the naturally associated subspaces of $d_j\times d_j$ ($j=1,2$) matrices
\beq\label{eq:spanSdefined}
\cS_j = \span \{I, A_{j,\ell}\colon \ell=1,\ldots,\tg\}= \span\{\cL_j(X)\colon  X\in\R^\tg\}
\subseteq
\mbS\R^{d_j\times d_j}.
\eeq

 We shall soon see that the condition
 $L_2$ dominates $L_1$, equivalently $\cD_{L_1}\subset \cD_{L_2},$
 is equivalent to a 
 property called complete positivity, defined below,
 of the unital linear mapping 
 $\tau:\cS_1\to \cS_2$ determined by
\begin{equation}
 \label{eq:define-tau}
  \tau(A_{1,\ell})=A_{2,\ell}.
\end{equation}
 
 A recurring theme in the
 non-commutative setting, such as that of
 a subspace of  C$^*$-algebra
 \cite{Ar69,Ar72,Ar08} or in
 free probability \cite{Vo04,Vo05}
 to give two of many examples, is the need to consider
 the {\bf complete matrix structure} afforded by
 tensoring  with $n\times n$ matrices (over positive integers $n$).
 The resulting theory of operator
 algebras, systems, spaces and matrix
 convex sets  has matured to the point that there are now
 several excellent books on the subject including
 \cite{BL04,Pau,Pi}.

 Let $\cT_j\subseteq\R^{d_j\times d_j}$ be unital linear subspaces
 closed under the transpose, and $\phi:\cT_1\to\cT_2$ a unital linear $*$-map.
 For $n\in\N$, $\phi$ induces the map
$$\phi_n=I_n\otimes\phi:\R^{n\times n}\otimes\cT_1=\cT_1^{n\times n}
 \to \cT_2^{n\times n},
 \quad M\otimes A \mapsto M \otimes \phi(A),
$$
 called an {\bf ampliation} of $\phi$.
 Equivalently,
$$
 \phi_n\left( \begin{bmatrix} T_{11} & \cdots & T_{1n} \\
 \vdots & \ddots & \vdots \\
 T_{n1} & \cdots & T_{nn}\end{bmatrix}\right)
 =\begin{bmatrix} \phi(T_{11}) & \cdots & \phi(T_{1n}) \\
 \vdots & \ddots & \vdots \\
 \phi(T_{n1}) & \cdots & \phi(T_{nn})\end{bmatrix}
$$
 for $T_{ij}\in\cT_1$.
 We say that $\phi$ is {\bf $k$-positive} if
 $\phi_k$ is a positive map. If $\phi$ is $k$-positive for every $k\in\N$,
 then $\phi$ is {\bf completely positive}.

\subsection{The Map $\tau$ is Completely Positive}
 \label{sec:tauisCP}
 A basic observation is that
 $n$-positivity of $\tau$ is equivalent to the inclusion
 $\cD_{L_1}(n)\subseteq\cD_{L_2}(n)$. Hence $\cD_{L_1}\subseteq\cD_{L_2}$
 is equivalent
 to complete positivity of $\tau,$ an observation which
 ultimately leads 
 to the algebraic characterization
 of Theorem \ref{thm:intro1}.

\begin{thm}
 \label{thm:tauisCP}
 Consider the monic linear pencils of equation \eqref{eq:penc12}
 and assume that
$\cD_{\cL_1}(1)$ is bounded.
 Let $\tau:\cS_1\to\cS_2$ be the unital linear map of equation
  \eqref{eq:define-tau}.
\ben[\rm (1)]
\item $\tau$ is $n$-positive if and only if
   $\cD_{L_1}(n)\subseteq \cD_{L_2}(n)$;
\item
  $\tau$ is completely positive if and only if
  $\cD_{L_1}\subseteq\cD_{L_2}$.
\een
\end{thm}
 
 Conversely, suppose $\cD$ is a unital $*$-subspace
 of $\SRdd$ and
 $\tau:\cD\to \mathbb S\mathbb R^{d^\prime\times d^\prime}$
 is completely positive.
 Given a basis $\{I,A_1,\ldots,A_g\}$ for $\cD$, let $B_j=\tau(A_j)$. Let
\[
  L_1(x) = I+\sum A_j x_j, \ \ \ L_2(x)=I+\sum B_j x_j.
\]
 The complete positivity of $\tau$ implies, if
  $L_1(X)\succ 0$, then $L_2(X)\succ 0$ and
 hence $\cD_{L_1}\subseteq \cD_{L_2}$.
 Hence the completely positive map $\tau$ (together with a choice
 of basis) gives  rise to an LMI domination.

\subsection{An Example}\label{sec:nonScalar}
 The following example illustrates the constructs of the
 previous two subsections. Let
$$L_1(x_1,x_2)=
 I+ \begin{bmatrix}  0 & 1 & 0 \\
                     1 & 0 & 0\\
                     0 & 0 & 0
                      \end{bmatrix} x_1 +
\begin{bmatrix}  0 & 0 & 1 \\
                     0 & 0 & 0\\
                     1 & 0 & 0
                      \end{bmatrix} x_2 =
\begin{bmatrix}  1 & x_1 & x_2 \\
                     x_1 & 1 & 0\\
                     x_2 & 0 & 1
                      \end{bmatrix}\in\mbS\R^{3\times 3}\ax
$$
and
$$
L_2(x_1,x_2)=I+ \begin{bmatrix}1  & 0 \\ 0 & -1
\end{bmatrix} x_1 +  \begin{bmatrix} 0  & 1 \\ 1 & 0
\end{bmatrix} y_2 =  \begin{bmatrix} 1+x_1  & x_2 \\ x_2 & 1-x_1
\end{bmatrix}\in\mbS\R^{2\times 2}\ax.
$$Then
\[
\begin{split}
\cD_{\cL_1} & =\{(X_1,X_2)\colon  I-X_1^2-X_2^2 \succ0\}, \\
\cD_{\cL_1}(1) & =\{(X_1,X_2)\in\R^2\colon  X_1^2+X_2^2< 1\}, \\
\cD_{\cL_2}(1) & =\{(X_1,X_2)\in\R^2\colon  X_1^2+X_2^2< 1\}.
\end{split}
\]
Thus $\cD_{\cL_1}(1)=\cD_{\cL_2}(1)$. On one hand,
$$
\left( \begin{bmatrix} \frac 12 & 0 \\ 0 & 0 \end{bmatrix},
\begin{bmatrix} 0& \frac 34  \\ \frac 34  & 0 \end{bmatrix}\right)
\in\cD_{L_1}\setminus\cD_{L_2},$$
so $\cL_1(X_1,X_2)\succ0$ does not imply $\cL_2(X_1,X_2)\succ0$.

On the other hand, $\cL_2(X_1,X_2)\succ0$ does imply
$\cL_1(X_1,X_2)\succ0$.
The map $\tau:\cS_2\to\cS_1$ in our example is given by
$$
\begin{bmatrix}1  & 0 \\ 0 & 1
\end{bmatrix}
\mapsto
\begin{bmatrix}  1 & 0 & 0 \\
                     0 & 1 & 0\\
                     0 & 0 & 1
                      \end{bmatrix},\quad
\begin{bmatrix}1  & 0 \\ 0 & -1
\end{bmatrix}
\mapsto
\begin{bmatrix}  0 & 1 & 0 \\
                     1 & 0 & 0\\
                     0 & 0 & 0
                      \end{bmatrix}, \quad
\begin{bmatrix} 0  & 1 \\ 1 & 0
\end{bmatrix}
\mapsto
\begin{bmatrix}  0 & 0 & 1 \\
                     0 & 0 & 0\\
                     1 & 0 & 0
                      \end{bmatrix}.
$$
Consider the extension of $\tau$ to a unital linear $*$-map
$\psi:\R^{2\times2}\to\R^{3\times3}$, defined by
$$
E_{11}\mapsto
\frac 12
\begin{bmatrix}  1 & 1 & 0 \\
                     1 & 1 & 0\\
                     0 & 0 & 1
                      \end{bmatrix}  , \;
E_{12}
\mapsto
\frac 12
\begin{bmatrix}  0 & 0 & 1 \\
                     0 & 0 & 1\\
                     1 & -1 & 0
                      \end{bmatrix}, \;
E_{21} \mapsto
\frac 12
\begin{bmatrix}  0 & 0 & 1 \\
                     0 & 0 & -1\\
                     1 & 1 & 0
                      \end{bmatrix} ,\;
E_{22}
\mapsto
\frac 12
\begin{bmatrix}  1 & -1 & 0 \\
                     -1 & 1 & 0\\
                     0 & 0 & 1
                      \end{bmatrix}.
$$
(Here $E_{ij}$ are the $2\times 2$ matrix units.)
 To show that $\psi$ is completely positive
 compute its Choi matrix defined as
\beq\label{eq:choi2x2}
C=\begin{bmatrix}
\psi(E_{11}) & \psi(E_{12})\\
\psi(E_{21}) & \psi(E_{22})
\end{bmatrix}.
\eeq
\cite[Theorem 3.14]{Pau}
 says $\psi$ is completely positive if and only if $C\succeq 0$.
 The Choi matrix is the key to
 computational algorithms in \cite[Section 4]{HKMlmi}.
 In the present case, to see that $C$ is positive semidefinite, note
$$
C=\frac 12 W^*W\quad \text{ for }\quad W=
\begin{bmatrix}
1 & 1 & 0 & 0 & 0 & 1 \\
 0 & 0 & 1 & 1 & -1 & 0 \\
\end{bmatrix}.
$$

 Now $\psi$ has a very nice
 representation:
\beq\label{ex:greatExample}
\psi(S)=\frac12 V_1^*SV_1+\frac 12V_2^*SV_2= \frac 12 \begin{bmatrix} V_1 \\ V_2\end{bmatrix}^*
\begin{bmatrix} S & 0\\ 0 & S\end{bmatrix}
\begin{bmatrix} V_1 \\ V_2\end{bmatrix}
\eeq
 for all $S\in\R^{2\times 2}$.
(Here
$
V_1=
\begin{bmatrix}
1 & 1 & 0  \\
 0 & 0 & 1 \\
\end{bmatrix}
$
and
$V_2=
\begin{bmatrix}
 0 & 0 & 1 \\
1 & -1 & 0 \\
\end{bmatrix},
$
thus
$
W = \begin{bmatrix} V_1 & V_2\end{bmatrix}.
$)
In particular,
\beq\label{ex:great2}
2L_1(x,y)= V_1^* L_2(x,y)V_1+V_2^* L_2(x,y)V_2.
\eeq
Hence $L_2(X_1,X_2)\succ0 $ implies $L_1(X_1,X_2)\succ0$, i.e.,
$\cD_{L_2}\subseteq\cD_{L_1}$.

 The computations leading up to equation \eqref{ex:great2}
 illustrate the proof of our linear Positivstellensatz,
 Theorem \ref{thm:intro1}.  For the details see 
\cite[Section 3.1]{HKMlmi}.

\subsection{Positivstellensatz on a Spectrahedron}
 Our non-commutative techniques 
 lead to
 a cleaner and more powerful commutative
 Putinar-type Positivstellensatz \cite{Put} for $p$ strictly
 positive on a bounded spectrahedron $\ocD_L(1)=\{x\in\R^g : L(x)\succeq0\}$.
 In the theorem which follows, $\SRdd\cy$ is the set  of symmetric
 $d\times d$ matrices with entries from
 $\mathbb R[y]$, the algebra of (commutative)
 polynomials with coefficients from $\mathbb R$.
 Note that an element of $\SRdd\cy$ may be identified
 with a polynomial (in commuting variables)
 with coefficients from $\SRdd$.

\begin{thm}\label{thm:commposIntro}
Suppose $L\in\SRdd\cy$ is a monic linear pencil
and $\ocD_\cL(1)$ is bounded. Then for every symmetric matrix polynomial
$p\in\R^{\ell\times\ell}\cy$ with
$p|_{\ocD_\cL(1)}\succ0$, there are
$A_j\in\R^{\ell\times\ell}\cy$,
and $B_k\in\R^{d\times\ell}\cy$
satisfying
\beq\label{eq:cpos1Intro}
p=\sum_j A_j^*A_j + \sum_k B_k^* \cL B_k.
\eeq
\end{thm}

 The Positivstellensatz, Theorem \ref{thm:commposIntro},
 has a non-commutative version
 for $\dd\times\dd$ matrix valued symmetric polynomials $p$
 in non-commuting variables positive on a nc LMI set $\ocD_L,$ see \cite{HKMlmi}.
 In the case this matrix valued polynomial $p$ is linear,
 this Positivstellensatz reduces to
 Theorem \ref{thm:intro1}, which can thus be regarded as a
``Linear Positivstellensatz''.
 For perspective we mention that the proofs of our Positivstellens\"atze
 actually rely on the linear Positivstellensatz.
 For experts we point out that the key
 reason LMI sets behave better is
 that the quadratic module associated to a
 monic linear pencil $L$ with \emph{bounded} $\ocD_L$ is \emph{archimedean}.

 We shall return to the topic of Positivstellens\"atze in
 Section \ref{sec:posss&nullss}.

\section{Non-commutative Convex semi-algebraic Sets are LMI Representable}\label{sec:conv=LMI}
 The main result of this section is that
   a bounded convex
 \ncboL semi-algebraic
 set 
 has a monic Linear Matrix Inequality
 representation.
 Applications
 and connections to semidefinite programming and
 linear systems engineering are discussed
 in Section \ref{subsec:motivate}.
  The work is also of interest in understanding a
  non-commutative (free) analog of convex
  semi-algebraic sets \cite{BCR}.

  For perspective, in the commutative case
  of a basic open semi-algebraic subset $\cC$
  of $\RR^{g}$, 
  there is a stringent condition, called the
 ``\emph{line test}'', which, in addition to convexity,
  is  necessary for $\cC$ to have an LMI representation.
  In two dimensions the line test is necessary and sufficient,
  \cite{HV07}, a result used by Lewis-Parrilo-Ramana \cite{LPR05}
  to settle a 1958 conjecture of Peter Lax on hyperbolic polynomials.
  Indeed LMI representations are closely tied to properties
  of hyperbolic polynomials; see this volume, the survey of Helton
and Nie.

 In summary, if a (commutative)
 bounded basic open semi-algebraic convex set
 has an LMI representation, then it must pass
 the highly restrictive line test; whereas
 a  \ncbbo semi-algebraic set has
 an LMI representation if and only if it is convex.

  A subset $\cS$ of $\smatng$ is closed under unitary conjugation
  if for every $X=(X_1,\dots,X_g)\in \cS$ and $U$ is a $n\times n$
  unitary, we have $U^*XU=(U^*X_1U,\dots, U^* X_gU)\in \cS.$
  The sequence $\cC=(\cC(n))_{n\in\N},$ where
  $\cC(n)\subseteq\smatng,$ is a {\bf non-commutative set}
  \index{non-commutative set}
  if it is closed under unitary conjugation
  and direct sums; i.e., if $X=(X_1,\dots,X_g)\in\cC(n)$
  and $Y=(Y_1,\dots,Y_g)\in\cC(m)$, then 
  $X\oplus Y= (X_1\oplus Y_1,\dots,X_g\oplus Y_g)\in\cC(n+m)$.
  Such
  set $\cC$
  has an {\bf LMI representation}
  \index{LMI representation}
  if there is a monic linear pencil $L$
  such that
$$
 \cC= \cD_L.
$$
  Of course, if $\cC=\cD_L$,
  then the closure $\overline \cC$ of $\cC$ has the representation
  $\{ X \colon L(X) \succeq 0\}$ and so we could also  refer to
  $\overline \cC$ as having an LMI representation.

  Clearly, if $\cC$ has an LMI representation, then
  $\cC$ is  a  convex \ncbo semi-algebraic set.
  The
  main result of this section is the converse, under the
  additional assumption that $\cC$ is bounded.

  Since we are dealing with matrix convex sets, it is not surprising
 that 
 the starting point for our analysis is the  matricial version
 of the Hahn-Banach Separation  Theorem of
 Effros and Winkler \cite{EW97} which
 says that given a point $x$
 not inside a matrix convex set there is a (finite) LMI which
 separates $x$ from the set.  For a general matrix convex
 set $\mathcal C$, the conclusion is then that there is a collection,
 likely infinite, of finite LMIs which cut out $\mathcal C$.

  In the case $\mathcal C$ is matrix convex and also semi-algebraic,
  the challenge 
  is to prove that there is actually
  a finite collection of (finite) LMIs which define $\mathcal C$.
  The techniques used to meet this challenge
  have little relation to previous work on
  convex non-commutative basic semi-algebraic sets.
  In particular, they do not involve non-commutative
  calculus and positivity. See \cite{HMlmi} for the details.

\subsection{Non-commutative Basic Open Semi-Algebraic  Sets}
 \label{subsec:semi-algebraic}
  Suppose  $p\in \ccP^{\dddd}$ is symmetric.
  In particular, $p(0)$ is a  $\dddd$ symmetric
  matrix. Assume that $p(0)\succ 0$.
  For each positive integer $n$, let
\[
  \posn = \{X\in\smatng \colon p(X)\succ0\},
\]
  and define $\pos$ to be the sequence
  (graded set) $(\posn)_{n=1}^\infty$.
  Let   $\cDpn$ denote the
 connected component of $0$ of $\posn$ and
 $\cD_p$ the sequence (graded set) $(\cDpn)_{n=1}^\infty$.
  We call $\cD_p$  {\bf the positivity  set} of $p$.
  \index{positivity set} \index{$\posn$} \index{$\cD_p$}
 In analogy with classical real algebraic geometry
 we call sets of the form $\cD_p$
 {\bf \ncbo semi-algebraic sets}.
 \index{semi-algebraic set, basic open non-commutative}
  (Note that
  it is not necessary to explicitly consider
  intersections of \ncbo semi-algebraic sets
  since the intersection $\cD_p \cap \cD_q$
  equals $\cD_{p\oplus q}$.)

\begin{remark}\rm
  By a simple affine linear
  change of variable the point $0$ can be
  replaced by $\lambda\in\mathbb R^g$.
  Replacing $0$ by a fixed $\Lambda\in\smatng$
  would require an extension of the theory.
\qed
\end{remark}

\subsection{Convex Semi-Algebraic Sets}
 \label{subsec:convex-basic}
 To say that $\cD_p$ is {\bf convex} \index{convex} means that each
  $\cDpn$ is convex (in the usual sense) and in this case
  we say $\cD_p$ is
  a {\bf convex \ncboL semi-algebraic set}.
 \index{convex non-commutative basic open non-commutative semi-algebraic set}
  In addition, we generally assume that
  $\cD_p$ is bounded; i.e.,
  there is a constant $K$ such for each $n$ and
  each $X\in \cD_p(n)$, we have
  $\|X\|=\sum \|X_j\|\le K.$
  \index{bounded}  Thus the following
  list of conditions summarizes our
  usual  assumptions on $p$.

\begin{assumption}
 \label{assume}
   Fix $p$ a $\dd \times \dd$  symmetric matrix
  of polynomials in $g$ non-commuting variables of  degree $d$.
  Our standard assumptions are:
 \begin{enumerate}[\rm (1)]
  \item $p(0)$ is positive definite;
  \item $\mathcal D_p$ is bounded; and
  \item $\mathcal D_p$ is convex.
 \end{enumerate}
\end{assumption}
 \index{Assumption \ref{assume}}

\subsection{The Result}
 \label{subsec:main-results}
  Our main theorem of this section is
\begin{thm}[\cite{HMlmi}]
 \label{thm:main}
   Every convex \ncbboL semi-algebraic set
   $($as in Assumption {\rm\ref{assume})}
   has an LMI representation.
\end{thm}

The proof of Theorem \ref{thm:main}  yields estimates
on the size of the representing LMI.

\begin{thm}
 \label{thm:mainQuant}
 Suppose $p$ satisfies the conditions
 of Assumption {\rm\ref{assume}}.
 Thus $p$ is a symmetric $\dd\times\dd$-matrix polynomial
 of degree $d$ in $g$ variables.
 Let $\nu=\dd \sum_{j=0}^d g^j$.
\ben[\rm (1)]
\item
 There
 is a $\mu\le \frac{\nu(\nu+1)}{2}$
 and a monic linear pencil $L\in\mbS\R^{\mu\times\mu}\ax$
such that
 $\cD_p=\cD_L$.
\item
  In the case that $p(0)=I_\dd$, the estimate on
  the size of the matrices 
in  $L$ reduces  to $\frac{\nus(\nus+1)}{2}$,
  where $\nus= \dd \sum_{j=0}^\Nd g^j$.
\een
\end{thm}
As usual,
 $\Nd $ stands for   the smallest integer $\geq 
\frac{d}{2}$. Of course
$$
  \Nd=\frac{d}{2} \text{ when $d$ is even}
  \quad \text{ and }\quad \Nd=\frac{d+1}{2} \text{ when $d$ is odd}.
$$
\index{$\Nd= \frac{d}{2}$   if $d$ even}
\index{$\Nd= \frac{d+1}{2}$ if $d$  odd}

The results above hold even if sets more general than $\cD_p$ are used.
Suppose $p(0)$ is invertible and
define $\cI_p$ to be the component of $\{ X \colon p(X)$ is invertible$\}$
containing 0. Then if $\cI_p$ is bounded and convex,
the theorems above still hold for $\cI_p$; it has an LMI representation.

An unexpected  consequence of Theorem \ref{thm:main} is that
projections of non-commutative semi-algebraic sets may not be semi-algebraic.
For details and proofs see \cite{HMlmi}.

\subsection{Motivation}\label{subsec:motivate}
 One of the main advances in systems engineering in the 1990's
 was the conversion of a set of problems to LMIs,
 since LMIs, up to modest size, can be solved numerically
 by semidefinite programs \cite{SIG97}.
 A large  class of linear
 systems problems are described in terms of
 a signal-flow diagram $\Sigma$ plus $L^2$
 constraints (such as energy dissipation).
 Routine methods convert such problems into
 a non-commutative polynomial inequalities
 of the form  $p(X)\succeq 0$ or $p(X)\succ 0$.

 Instantiating
 specific  systems of linear differential
 equations  for the ``boxes'' in the system flow diagram
 amounts to substituting  their coefficient matrices
 for variables in the polynomial $p$.
 Any property  asserted to  be true must hold
 when matrices of any size are substituted into $p$.
 Such problems are referred to as dimension free.
 We emphasize, the polynomial $p$ itself is determined
 by the signal-flow diagram $\Sigma$.

Engineers vigorously seek convexity, since optima are global
and convexity lends itself to numerics.
Indeed,  there are over  a thousand papers trying to convert
linear systems problems to convex ones and the only known technique is
the rather blunt trial and error instrument of trying to guess an LMI.
Since  having an LMI is seemingly  more restrictive than convexity,
there has been the hope, indeed expectation, that some practical
class of convex situations   has been missed.
  The problem solved here
  (though not operating at full engineering generality,
  see \cite{HHLM08})
  is a paradigm for the type of algebra occurring
  in systems problems governed by signal-flow diagrams;
 such physical problems directly present non-commutative semi-algebraic sets.
  Theorem \ref{thm:main} gives compelling evidence
  that all such convex situations are associated to some LMI.
Thus we think the implications of our results here are negative
for linear systems engineering; for dimension free problems
there is no convexity beyond LMIs.

  A basic question regarding the range of   applicability of SDP is:
  which sets have an LMI representation?
  Theorem \ref{thm:main} settles, to a reasonable extent, the case where
 the variables are non-commutative
 (effectively dimension free matrices).

\section{Convex Polynomials}
\label{sec:convexPolys}
 We turn now from non-commutative convex sets to
 non-commutative convex polynomials. 
  If $p$ is concave ($-p$ is convex) and monic, then 
  the set $S=\{X:p(X)\succ 0\}$ is a convex \ncboL. If it
  is also bounded, then, by the results of the previous
  section, it has an LMI representation.  However, much
  more is true and the analysis turns on a nc
  version of the Hessian and connects with nc semi-algebraic
  geometry.

 A symmetric polynomial $p$ is {\bf matrix convex},
 \index{convex polynomial} or simply {\bf convex} for short,
  if for each positive integer $n$, each
 pair of tuples $X\in\smatng$ and $Y\in\smatng$, and each $0\le t \le
 1$,
\begin{equation}
 \label{eqn:matrixconvex} p(tX+(1-t)Y)\preceq tp(X)+(1-t)p(Y).
\end{equation}
 Even in one-variable, convexity in the non-commutative setting
 differs from convexity in the commuting case because here $Y$ need
 not commute with $X$. For example, to see  that the polynomial
 $p=x^4$ is not matrix convex, let
$$
 X=\begin{bmatrix} 4&2\\2&2\end{bmatrix}
  \mbox{ and }
 Y=\begin{bmatrix} 2&0\\0&0\end{bmatrix}
$$
 and compute
$$
 \frac12 X^4+\frac12 Y^4 - \left ( \frac12 X+\frac12 Y \right)^4 =
 \begin{bmatrix}164&120\\120&84\end{bmatrix}
$$
 which is not positive semidefinite. On the other hand,  to verify
 that $x^2$ is a matrix convex polynomial, observe that
\[
  tX^2+(1-t)Y^2 - (tX+(1-t)Y)^2
=t(1-t)(X-Y)^2 \succeq 0.
\]

 It is possible to automate checking for convexity, rather
 than depending upon lucky choices of $X$ and $Y$ as was
 done above.
 The theory described in \cite{CHSY03},
 leads to and validates
 a symbolic algorithm for determining regions of convexity
 of \NC polynomials and even 
 of \NC rational functions (for \NC rationals see \cite{KVV09,HMV06})
 which is implemented in NCAlgebra.

 Let us illustrate it on the example $p(x)=x^4$.
 The NCAlgebra command is

\centerline{{\bf NCConvexityRegion}[Function $F$, \{Variables $x$\}].}

\begin{verbatim}
In[1]:=  SetNonCommutative[x];
In[2]:=  NCConvexityRegion[ x**x**x**x, {x} ]
Out[2]:= { {2, 0, 0},  {0, 2}, {0, -2} }
\end{verbatim}

\noindent
  which we interpret as
  saying that $p(x)=x^4$ is convex on the set of matrices
  $X$ for which the
  the $3\times 3$ block matrix valued non-commutative function
\begin{equation}
\label{eq:x4diag}
  \rho(X)= \begin{bmatrix}
      2& 0 & 0 \\ 0 & 0 & 2 \\ 0 & -2 & 0
   \end{bmatrix}
\end{equation}
  is positive semidefinite.
  Since $\rho(X)$ is constant and never positive semidefinite,
  we conclude that $p$ is {\it nowhere} convex.

 This example is a simple special case of the following theorem.

\begin{thm}[{\rm\protect{\cite{HM04conv}}}]
\label{thm:matrixconvexI}
   Every convex symmetric polynomial
   in the free algebra $\RRx$
    has degree two or less.
\end{thm}
\index{convex degree 2 Theorem}

\subsection{The Proof of Theorem \ref{thm:matrixconvexI} and its
 Ingredients}
  Just as in the commutative case, convexity of a
  symmetric $p\in \RRx$ is
  equivalent to positivity of its Hessian
  $q(x)[h]$ which is a polynomial
  in the $2g$ variables $x=(x_1,\ldots,x_g)$
  and $h=(h_1,\ldots,h_g)$.
  Unlike the commutative case, a positive non-commutative
  polynomial is a sum of squares.  
  Thus, if $p$ is convex, then its Hessian $q(x)[h]$
  is a sum of squares.  Combinatorial
  considerations say that a Hessian which
  is also a sum of squares must come from
  a polynomial of degree two.
  In the remainder
  of this section we flesh out this argument,
  introducing the needed definitions, techniques,  and results.

\subsubsection{Non-commutative Derivatives}
 For practical purposes, the 
 {\bf $k^{th}$-directional derivative} of a nc polynomial $p$ is given 
 \index{directional derivative}
 by
$$
 p^{(k)}(x)[h]= \frac{d^k}{dt^k} p(x+th) \Big|_{t = 0}.
$$
 Note that $p^{(k)}(x)[h]$ is homogeneous of degree $k$ in $h$
  and moreover, if $p$ is symmetric so is $p^{k}(x)[h]$.
  For $X,H\in\smatng$ observe that
   \begin{equation*}
     p^\prime(X)[H]=\lim_{t\to 0} \frac{p(X+tH)-p(X)}{t}.
   \end{equation*}

\begin{example}\rm
 \label{ex:two}
 The one variable $p(x)=x^4$ has first derivative
$$
 p^{\prime}(x)[h]= hxxx + xhxx  + xxhx + xxxh.
$$
 Note each term is linear in $h$ and $h$ replaces each occurrence
 of $x$ once and only once.
 The Hessian, or second derivative, of $p$ is
\[
 p^{\prime\prime}(x)[h]=
   2hhxx   + 2hxhx + 2 hxxh
  + 2xhhx   + 2 xhxh  + 2 xxhh.
\]
 Note each term is degree two  in $h$ and $h$ replaces each pair of
 $x$'s exactly once. 
\end{example}

\begin{thm} [\cite{HP07}]
\label{thm:matrixderivI}
   Every  symmetric polynomial $p\in\RRx$
   whose $k^{th}$ derivative is a matrix positive polynomial
    has degree $k$ or less.
\end{thm}

\proof \ See \cite{HP07} for the full proof
  or \cite{HM04conv} for case of $k=2$.
  The very intuitive proof based upon
  a little non-commutative semi-algebraic
  geometry is sketched in the next subsection.
\qed

\subsubsection{A Little \NCRAG}
  The proof of Theorem \ref{thm:matrixconvexI} 
  employs the most fundamental of all 
  non-commutative Positivstellens\"atze.

   A symmetric non-commutative polynomial  $p$ is
   {\bf matrix positive} or simply {\bf positive}
    provided
   $p(X_1, \ldots ,X_g)$ is
   positive semidefinite
   for every $X\in\gtupn$ (and every $n$).
 An example of a matrix positive polynomial
 is a {\bf Sum of Squares}
 of  polynomials, meaning an expression of the
 form
$$
 p(x)=
  \sum_{j=1}^c h_j(x)^\TT  h_j(x).
$$
 Substituting $X \in \gtupn$ gives
$
  p(X)=\sum_{j=1}^c h_j(X)^\TT  h_j(X) \succeq 0.
$
Thus $p$ is positive.  Remarkably these are the
only  positive non-commutative polynomials.

\begin{thm}[\cite{H}]
 \label{thm:posPoly}
   Every matrix positive  polynomial is a sum of squares.
\end{thm}

 This theorem is just a sample of the structure of \NC semi-algebraic
 geometry, the topic of Section \ref{sec:posss&nullss}.

 Suppose  $p\in\RRx$ is (symmetric and) convex
 and  $Z,H\in\smatng$ and $t\in\mathbb R$ are given.
 In the definition of convex, choosing $X=Z+tH$
 and $Y=Z-tH$, it follows that
\begin{equation*}
  0\preceq p(Z+tH)+p(Z-tH) -2 p(Z),
\end{equation*}
  and therefore
\begin{equation*}
  0\preceq \lim_{t \rightarrow 0}\frac{p(X+tH)+p(X-tH)-2p(X)}{t^2}= p^{\prime\prime}(X)[H].
\end{equation*}
 Thus the Hessian of $p$ is matrix positive
 and since, in the non-commutative setting,
 positive polynomials are sums of squares
 we obtain the following theorem.

\begin{prop}
 \label{prop:posPoly}
  If $p$ is matrix convex, then its Hessian $p^{\prime\prime}(x)[h]$
  is a sum of squares.
\end{prop}

\subsubsection{Proof of Theorem {\rm\ref{thm:matrixconvexI}} by example}
  Here we illustrate the proof of Theorem \ref{thm:matrixconvexI}
  based upon
  Proposition \ref{prop:posPoly} by showing that 
  $p(x)=x^4$ is not matrix convex. Indeed, if $p(x)$ is matrix convex, then
  ${p^{\prime\prime}(x)}[h]$ is matrix positive and therefore,
  by Proposition \ref{prop:posPoly}, there exists
  a $\ell$ and polynomials $f_1(x,h),\ldots,f_\ell(x,h)$ such that
\begin{equation*}
 \begin{split}
  \frac 12p^{\prime\prime}(x)[h]&=  hhxx + hxhx +  hxxh + xhhx+ xhxh + xxhh \\
  &=  f_1(x,h)^\TT f_1(x,h) + \cdots + f_\ell(x,h)^\TT f_\ell(x,h).
 \end{split}
\end{equation*}
  One can show that each $f_j(x,h)$ is linear in $h$.
  On the other hand,
  some term $f_i^\TT f_i$ contains $hhxx$ and thus
  $f_i$ contains $hx^2$.  Let $m$ denote
  the largest $\ell$ such that some  $f_j$ contains the term $h
 x^\ell.$
  Then $m\ge 1$ and for such $j$, the product
$f_j^\TT f_j$ contains the term $h x^{2m}h$
  which cannot be cancelled out, a contradiction. \qed

 The proof of the
 more general, order $k$ derivative,
 is similar, see \cite{HP07}.

\subsection{Non-commutative Rational and Analytic  Functions}

A class of functions bigger than nc polynomials is given by nc analytic
functions, see e.g.~ Voiculescu \cite{Vo04,Voprept}  
or the forthcoming paper of 
Kaliuzhnyi-Verbovetskyi and Vinnikov for an introduction.
The rigidity of nc bianalytic maps is investigated by Popescu \cite{Po10}; see
also \cite{HKMS09,HKMprept,HKMfree}.
For other properties of nc analytic functions, a very interesting body
of work, e.g.~by Popescu \cite{PoAdv} can be used as a gateway.

The articles \cite{BGM06a,KVV09,HMV06}
deal with non-commutative rational functions.
For instance, \cite{HMV06} shows that if a non-commutative rational function
is convex in an open set, then it is the Schur Complement of some
monic linear pencil.

\section{Algebraic Certificates of Positivity}
\label{sec:posss&nullss}

In this section we give a brief overview 
of various
free $*$-algebra analogs  to the classical 
Positivstellens\"atze, i.e., 
theorems
characterizing polynomial inequalities in a purely
algebraic way.
Here it is of benefit to consider free \emph{non-symmetric} variables.
That is, let $x=(x_1,\ldots,x_g)$ be non-commuting variables
and $x^*=(x_1^*,\ldots,x_g^*)$ another set of non-commuting variables.
Then 
   $\R\axs$ \index{$\R\axs$} is the free $*$-algebra of polynomials in the
   non-commuting indeterminates $x,x^*.$

There is a natural involution $*$ on $\R\axs$ induced by
$x_i\mapsto x_i^*$ and $x_j^*\mapsto x_j$. As before,
$p\in\R\axs$ is symmetric if $p=p^*$. An element of
the form $p^*p$ is a square, and $\Sigma^2$ denotes
the convex cone of all sums of squares.
Given a matrix polynomial $p=\sum_w p_w w\in\R\axs^{\dd\times\ddp}$ and
$X\in(\Rnn)^g$, we define the evaluation
 $p(X,X^*)$ by analogy with evaluation in the symmetric variable case.

\subsection{Positivstellens\"atze}
\label{subsec:Positivstellensatz}
 This subsection gives an indication of various
 free $*$-algebra analogs  to the classical theorems
 characterizing polynomial inequalities in a purely
 algebraic way.
 We will start by sketching a proof of the following
 refinement of Theorem \ref{thm:posPoly}.

\begin{thm}[\cite{H}]
\label{thm:free SOS} \index{Sum of Squares Theorem, Free}
Let $p \in
\FF_d$ be a non-commutative polynomial. If
 $p(M,M^\T) \succeq 0$ for all $g$-tuples of linear operators $M$ acting
 on a
Hilbert space of dimension at most $N(k):=\dim \R\axs_k$ with $2k \geq d+2$,
then $ p \in \Sigma^2$.
\end{thm}

\begin{proof} Note that a polynomial $p$
satisfying the hypothesis automatically satisfies $p=p^*$.
The only necessary technical result we need is the closedness of
the cone $\Sigma^2_k$ in the Euclidean topology of the finite
dimensional space $\FF_k$. This is done as in the commutative
case, using Carath\'eodory's convex hull theorem,
more exactly,
every polynomial of $\Sigma^2_k$ is a convex combination of at most
$\dim \FF_k +1$ squares (of polynomials). On the other hand the
 positive functionals on $\Sigma^2_k$  separate the
points of $\FF_k$.
See for details \cite{HMP1}.

Assume that $p \notin \Sigma^2$ and let $k \geq (d+2)/2$, so that
$p \in \FF_{2k-2}$. Once we know that $\Sigma_{2k}^2$ is a closed
cone, we can invoke Minkowski separation theorem and find a
symmetric functional $L \in \FF_{2k}'$ providing the strict separation:
$$
L(p)<0 \leq L(f), \quad f \in \Sigma_{2k}^2.$$
Applying the Gelfand-Naimark-Segal construction to $L$ yields
 a tuple $M$ of operators acting on a
Hilbert space $H$ of dimension $N(k)$ and a vector $\xi \in H$,
such that
$$ 0 \leq  \langle p(M,M^\T) \xi, \xi \rangle = L(p) <0,$$
a contradiction.  
\end{proof}

When compared to the commutative framework, this theorem is
stronger in the sense that it does not assume a strict positivity
of $p$ on a well chosen ``spectrum''. Variants with supports (for
instance for spherical tuples
$M: \ M_1^\T M_1 + ...+ M_g^\T M_g \preceq I$)
of the above result are discussed in \cite{HMP1}.

To draw a very general conclusion from the above computations:
when dealing with positivity in a free $*$-algebra, the standard point
evaluations (or more precisely prime or real spectrum evaluations) of the
commutative case are replaced by matrix evaluations of the free
variables.
The positivity can be tailored
to ``evaluations in a
supporting set''. The results pertaining to the resulting algebraic
decompositions are called Positivstellens\"atze, see 
\cite{PD} for details in the commutative setting. We state below
an illustrative and generic result, from \cite{HM04pos}, for sums
of squares decompositions in a free $\ast$-algebra.

\begin{thm}[\cite{HM04pos}]
\label{thm:posspq}
\index{Free Positivstellensatz, with supports}
Let $p=p^* \in \FF$
and let $q = \{ q_1,...,q_k\} \subseteq \FF$ be a set of
symmetric
polynomials, so that
$$\QM(q) = {\rm co} \{ f^\T q_i f; \ f \in \FF, \ 0 \leq i \leq
k\}, \ q_0 =1,$$ contains $1-x_1^\T x_1 -...-x_g^\T x_g$ . If
for all tuples of linear bounded Hilbert space operators $X =
(X_1,...,X_g)$, we have 
\beq
\label{eq:qpos}
 q_i(X,X^\T) \succeq 0, \ 1 \leq i \leq k \quad \Rightarrow\quad
 p(X,X^\T)  \succ 0,
\eeq
then $p \in \QM(q)$.
\end{thm}

Henceforth, call $\QM(q)$ the {\bf quadratic module} generated
by the set of polynomials $q$.

We omit
the proof of Theorem \ref{thm:posspq}, as it is very
similar to the previous proof. The only difference is in
the separation theorem applied. For details, see \cite{HM04pos}.

Some interpretation is needed in degenerate cases, such as those
where no bounded operators satisfy the relations
$q_i(X,X^\T) \succeq 0$.
Suppose for example, if  $\phi$ denotes the defining relations
for the Weyl algebra and the $q_i$ include $- \phi^* \phi$.
In this case, we would say $ p(X,X^\T) \succ 0$,
since there are no $X$ satisfying $q(X,X^\T)$, and
voila $p \in \QM(q)$ as  the theorem says.
A non-archimedean Positivstellensatz for the Weyl algebra,
which treats unbounded representations and eigenvalues of polynomial
partial differential operators, is given in  
\cite{Scmu:05}.

A paradigm practical question with matrix inequalities is:\\

{\it
 Given a non-commutative symmetric polynomial $p(a,x)$ and
 a $n \times n$ matrix tuple $A$, find $X \succeq 0$
 if possible which makes $ p(A,X) \succeq 0$.}\\

 As a refinement of this problem, let $q(a,x)$ be 
 a given nc symmetric polynomial. For a given
 $A$, find $X$ if possible, such that both
 $q(A,X)$ and $p(A,X)$ are positive semidefinite.
 The infeasibility of this latter problem is 
 equivalent to the statement, if $q(A,X)\succeq 0$, then
 $p(A,X)\not\succeq 0$. 
 There is keen interest in numerical solutions 
 of such problems. 
 The next theorem informs us that
 the main issue is the matrix coefficients $A$,
 as it gives a ``certificate of
 infeasibility'' for the problem in the absence of $A$. 

\begin{thm}[The Nirgendsnegativsemidefinitheitsstellensatz
\cite{KS07}]\label{thm:nicht-nsd}
 Let $p=p^* \in \FF$
 and let $q = \{ q_1,...,q_k\} \subset \FF$ be a set of
 symmetric
 polynomials, so that $\QM(q)$
 contains $1-x_1^\T x_1 -...-x_g^\T x_g$ . If
 for all tuples of linear bounded Hilbert space operators $X =
 (X_1,...,X_g)$, we have
\beq
\label{eq:npos}
 q_i(X,X^\T) \succeq 0, \ 1 \leq i \leq k \quad \Rightarrow\quad
 p(X,X^\T)  \not\preceq 0,
\eeq
 then there exists an integer  $r $ and
 $h_1,\ldots, h_r \in \FF$ with $\sum_{i=1}^r h_i^*p h_i\in 1+\QM(q)$.
\end{thm}

\begin{proof}
 By \eqref{eq:npos},
$$
 \{X\mid q_i(X,X^\T) \succeq 0, \ 1 \leq i \leq k,\, - p(X,X^\T)\succeq0\}=
 \varnothing.
$$
Hence $-1 \in \QM(q,-p)$ by Theorem \ref{thm:posspq}. 
\end{proof}

\subsection{Quotient Algebras}
\label{quotients}
The results from Section \ref{subsec:Positivstellensatz}
allow a variety of specializations to quotient
algebras.
In this subsection we consider  a two sided ideal
$\cI$  of $\FF$
which need not be invariant under $\ast$.
Then one can replace the quadratic module $\QM$ in the statement of
a Positivstellensatz with
$\QM(q) + \cI$,
and apply similar arguments as above.
For instance, the next simple observation can be deduced.

\begin{cor}
Assume, in the hypotheses of
  Theorem {\rm \ref{thm:posspq}},
  that the relations \eqref{eq:qpos} include
some relations of the form $r(X,X^\T) =0$,
even with $r$ not symmetric,
then
\beq
p \in \QM(q) + \cI_{r}
\eeq
where $\cI_{r}$ denotes the two sided
ideal generated by ${r}$.
\end{cor}

\begin{proof}
This follows immediately from
$p \in \QM(q, \; - r^* r)$
which is a consequence of Theorem \ref{thm:posspq}
and the fact
\[
\QM(q, \; - r^* r)  \subset \QM(q)+\cI_{r}.\qedhere
\]
\end{proof}

For instance, we can look at the situation where
 $r$ is the commutator $[x_i,x_j]$
as insisting on positivity of $ q(X)$ only on commuting
tuples of operators, in which case
 the ideal $\cI$ generated by $[x_j^*,x_i^*], \ [x_i,x_j]$
is added to $\QM(q)$.
The classical commuting case is captured by the
corollary applied to the ``commutator
ideal'':
$\cI_{[x_j^*,x_i^*], \ [x_i,x_j], \ [x_i,x_j^*] }$
for $i,j  = \ 1, \ldots, g$
which requires testing only on
commuting tuples of operators drawn from a commuting $C^*$-algebra.
The classical Spectral Theorem,  then converts this
to testing only on $\C^g$, cf.~\cite{HP07}.

The situation where one tests for constrained positivity
in the absence of an archimedean property is thoroughly
analyzed in \cite{Scmu:09}.

\subsection{A Nullstellensatz}
With similar techniques (well chosen, separating,
$\ast$-representations of the free algebra) and a
rather different ``dilation type'' of argument, one can prove a series
of Nullstellens\"atze.

We state for information one of them.
For an early version see \cite{HMP2}.

\begin{thm}
\label{thm:NullSS}
  Let $\pp_1(x),...,\pp_m(x) \in \RRx$
  be polynomials not depending on the $x_j^\T$ variables and let
  $\qq(x,x^\T) \in \FF$. Assume that for every $g$ tuple $X$ of linear
  operators acting on a finite dimensional Hilbert space $H$, and
every vector $v \in   H$, we have:
  $$ (\pp_j(X)v = 0, \ 1 \leq j \leq m) \ \Rightarrow \ (\qq(X,X^\T)v =
  0).$$
  Then $\qq$ belongs to the left ideal $\FF \pp_1 + ... + \FF \pp_m$.
\end{thm}

Again, this proposition is stronger than its commutative
counterpart. For instance there is no need of taking higher powers
of $\qq$, or of adding a sum of squares to $\qq$.  Note that
here $\RRx$ has a different meaning than earlier, since,
unlike previously,
the variables are nonsymmetric.

We refer the reader to \cite{HMP3} for the proof of Theorem
\ref{thm:NullSS}. 
An earlier, transpose-free Nullstellensatz due to Bergman was
given in 
\cite{HM04pos}.

\medskip

Here is   a theorem which could be regarded as
a very different type of \NC  Nullstellensatz.

\begin{thm}[\protect{\cite[Theorem 2.1]{KS08}}]\label{thm:trace0}
Let $p=p^* \in \FF_d$ be a non-commutative polynomial satisfying
 $ {\rm tr}\, p(M,M^\T) = 0$ for all $g$-tuples of linear operators $M$
acting  on a Hilbert space of dimension at most $d$.
Then $ p$ is a sum of commutators of non-commutative polynomials.
\end{thm}

\smallskip

We end this subsection with an example which goes against any
intuition we would carry from the commutative case, see
\cite{HM04pos}.

\begin{example}
  \label{ex:unreal}
Let $\pp=(x^\T x+xx^\T)^2$ and $\qq=x+x^\T$ where $x$ is a
single variable. Then, for every matrix $X$ and vector $v$
(belonging to the space where $X$ acts), $\pp(X)v=0$ implies
$\qq(X)v=0$; however,
  there does not exist
  a positive integer $m$ and $r,r_j \in \mathbb R \langle x,
x^\T\rangle,$ so that
  \begin{equation}
   \label{eq:unrealrep}
     \qq^{2m}+\sum r_j^\T r_j =  \pp r + r^\T \pp .
  \end{equation}
Moreover,  we can modify the example to add the condition $\pp(X)$
is  positive semidefinite  implies $\qq(X)$ is  positive
semidefinite and still not obtain this representation.
\end{example}

\subsection{Tracial Positivstellensatz}

Another type of non-commutative positivity is given by the trace.
A polynomial $p\in\FF$ is called {\bf trace-positive} if
$\tr p(X,X^*)\geq0$ for all $X\in(\Rnn)^g$.
The main motivation for studying these comes from two
outstanding open problems: Connes' embedding
conjecture \cite{Con} from operator algebras \cite{KS08}
and the Bessis-Moussa-Villani (BMV) conjecture \cite{bmv} from
quantum statistical mechanics \cite{KSbmv}.

Clearly, a sum of a matrix positive (i.e., sum of hermitian squares
by Theorem \ref{thm:free SOS}) and a trace-zero (i.e., sum of commutators by
Theorem \ref{thm:trace0}) polynomial is trace-positive. However, unlike in
the matrix positive case, not every trace-positive polynomial 
is of this form  \cite{KS08,KSbmv}.

\begin{example}
Let $x$ denote a single non-symmetric variable and
\begin{equation*}
M_0:= 3 x^4 - 3 (xx^*)^2 - 4 x^5x^* - 2 x^3 (x^*)^3   +  2 x^2x^*x(x^*)^2 
+ 2 x^2 (x^*)^2xx^* + 2 (xx^*)^3.
\end{equation*}
Then the non-commutative Motzkin polynomial in non-symmetric variables
is
$$
M:=1+M_0+M_0^*\in\FF.
$$
It is trace-positive but
is not a sum of hermitian squares and commutators.
\end{example}

Life is somewhat easier in the constrained, bounded case.
For instance, in the language of operator algebras
we have:

\begin{thm}[\cite{KS08}]
For $f=f^*\in\FF$ the following are equivalent:
\begin{enumerate}[{\rm(i)}]
\item $\tr\big(f(a,a^*)\big)\ge 0$ for all 
finite von Neumann algebras $\cA$ and 
all tuples of contractions $a\in\cA^g$;
\item for every $\ep\in\R_{>0}$, $f+\ep$ is a sum of commutators and of an
element from\\ $\QM(1- x_1^*x_1,\ldots,1-x_g^*x_g)$.
\end{enumerate}
\end{thm}
\noindent
The big open question \cite{Con,KS08} is whether (i) or (ii) is equivalent
to 
\begin{enumerate}[{\rm(i)}]
\item[\rm (iii)]
$\tr\big(f(X,X^*)\big)\ge 0$ for all $n\in\N$ and
all tuples of contractions $X\in(\Rnn)^g$.
\end{enumerate}

An attempt at better understanding trace-positivity is made
in \cite{BK}, where the duality between trace-positive polynomials
and the \emph{tracial moment problem} is exploited.
The tracial moment problem is the following question:
For which sequences $(y_w)$ indexed by words $w$
in $\x,\x^*$,
does there exist $n\in\N$, and a
positive
Borel measure $\mu$ on $(\R^{n\times n})^{g}$ 
satisfying
\begin{equation}\label{eq:gewidmetmarkus}
y_w = \int \! w( A,A^*) \, d\mu( A)\, ?
\end{equation}
Such a sequence is a \emph{tracial moment sequence}.
If one is interested only in \emph{finite} sequences $(y_w)$, then
this is the \emph{truncated} tracial moment problem.

To a sequence $y=(y_w)$ one associated the (infinite) \emph{Hankel matrix}
$M(y)$, indexed by words, by $M(y)_{u,v}=y_{u^*v}$. One of the results
in \cite{BK} shows that if $M(y)$ is positive semidefinite and of 
finite rank, then $y$ is a tracial moment sequence. In the truncated
case a condition called ``flatness'' governs the existence of a representing
measure, much like in the classical case. For details and proofs 
see \cite{BK}.

For the free non-commutative moment problem we refer the reader to
\cite{Mc:01}.

\section{Algebraic Software}
\label{sec:softw}

This section briefly surveys existing software
dealing with non-commutative convexity (Section
\ref{subsec:ncalg}) and positivity (Section \ref{subsec:ncsos}).

\subsection{NCAlgebra under Mathematica}\label{subsec:ncalg}
Here is a list of  software running under NCAlgebra \cite{HdOSM10}
(which runs under Mathematica)
that implements and experiments on symbolic algorithms
pertaining to non-commutative Convexity and LMIs.

NCAlgebra is available from
\url{http://www.math.ucsd.edu/~ncalg}

\begin{itemize}
\item
{\bf Convexity Checker.} \
\label{ConvC}
Camino, Helton, Skelton, Ye \cite{CHSY03} have an (algebraic) algorithm for
determining the region on which a rational expression is convex.
\item
{\bf Classical Production of LMIs.} \
There are two Mathematica NCAlgebra notebooks by
de Oliveira and Helton.
The first is based on
algorithms for implementing the  1997 approach
of Skelton, Iwasaki and  Grigonidas \cite{SIG97}
associating LMIs to
more than a dozen
control problems.
The second (requires C++ and NCGB)
 produces LMIs by symbolically implementing the 1997 change of
variables method of Scherer et al.
\item
{\bf Schur Complement  Representations of a non-commutative rational.} \
This computes a linear pencil whose Schur complement 
is the given nc rational function $p$ using 
Shopple - Slinglend  thesis algorithm.
It is not known if $p$ convex near $0$ always
leads to a monic pencil via this algorithm, but
we never saw a counter example.
\item
{\bf Determinantal Representations.} \
Finds Determinantal Representations of a given polynomial $p$.
Shopple - Slinglend implement
Slinglend's thesis algorithm  plus the \cite{HMV06} algorithm.
Requires NCAlgebra.
\end{itemize}

See
\url{http://www.math.ucsd.edu/~ncalg/surveydemo}
for occurences of available demos.

\subsection{NCSOStools under Matlab}\label{subsec:ncsos}
NCSOStools \cite{ckp10}
which runs under Matlab,
implements and experiments on numeric algorithms
pertaining to non-commutative positivity and sums of squares.
Here is a sample of features available.

\begin{itemize}
\item {\bf Non-commuting variables.} \
Basic symbolic computation with nc variables for Matlab
has been implemented.
\item {\bf Matrix-positivity.}\ 
An nc polynomial $p$ is matrix positive if and only if it is a sum
of squares. This can be easily tested using a variant of the classical
Gram matrix method. Indeed, $p\in\R\ax_{2d}$ is a sum of squares if
and only if 
$p=\ax_d^* G \ax_d$ for a positive semidefinite $G$. (Here, 
$\ax_d$ denotes a (column) vector of all words in $x$ of degree $\leq d$.)
This can be easily formulated as a feasibility semidefinite program (SDP).
\item {\bf Eigenvalue optimization.}\ 
Again, using SDP we can compute
the smallest eigenvalue $f^\star$ a symmetric $f\in\R\ax$ can attain.
That is,
\begin{equation*}
f^\star=\inf\big\{ \langle f(A)v,v\rangle :\ 
A \text{ a $g$-tuple of symmetric matrices, } v \text{ a unit vector}\big\}.
\end{equation*}
Hence $f^\star$ is the greatest lower bound on the eigenvalues
$f(A)$ can attain for $g$-tuples of symmetric matrices $A$, i.e.,
$(f-f^\star)(A)\succeq 0$ for all $n$-tuples of symmetric matrices $A$,
and $f^\star$ is the largest real number with this property.
Given that a polynomial is matrix positive if and only if it
is a sum of squares
we can compute $f^\star$ efficiently with SDP:
\begin{equation*}
\label{eq:mineig}
\begin{array}{rcl}
   f^{\star} & = & \sup ~~             \la \\
   \mbox{s.\,t.} & &  f-\la~\in~\Sigma^2.
\end{array}
\end{equation*}

\item {\bf Minimizer extraction.}\ 
Unlike in the commutative case, if $f^\star$ is attained, then
minimizers $(A,v)$ can always be computed.
That is, $A$ is a $g$-tuple of symmetric matrices
and $v$ is a unit eigenvector for $f(A)$ satisfying
\begin{equation*}
f^\star=\langle f(A)v,v\rangle.
\end{equation*}
Of course, in general $f$ will not be bounded from below.
Another problem is that even if $f$ is bounded, the
infimum $f^\star$ need not be attained.
The core ingredient of this minimizer extraction is
the nc moment problem governed by a condition
calles ``\emph{flatness}'',
together with the GNS construction.

\item{\bf Commutators and Cyclic equivalence.}\ 
Two polynomials are cyclically equivalent if their difference
is a sum of commutators. This is easy to check.

\item{\bf Trace-positivity.}\ 
The sufficient condition for trace-positivity (i.e., sum of squares
up to cyclic equivalence) is tested for using a variant of the
Gram matrix method applied to matrix positivity.
\end{itemize}

NCSOStools is extensively documented and available at
\url{http://ncsostools.fis.unm.si}

\linespread{1.2}

\end{document}